\documentclass[12pt]{amsart}
\usepackage{amsmath,amsfonts,amssymb, bm}

\usepackage{graphicx,eepic}
\newcommand{\BN}{\mathbb{N}}

\usepackage[normalem]{ulem}

\newtheorem{lemma}{Lemma}[section]

\newtheorem{prop}[lemma]{Proposition}
\newtheorem{thm}[lemma]{Theorem}
\newtheorem{cor}[lemma]{Corollary}
\theoremstyle{definition}
\newtheorem{Def}[lemma]{Definition}
\newtheorem{example}[lemma]{Example}

\theoremstyle{remark}
\newtheorem{rmk}[lemma]{Remark}

\numberwithin{equation}{section} \numberwithin{table}{section}

\newcommand{\beq}{\begin{equation}}
\newcommand{\eeq}{\end{equation}}

\newcommand{\ov}{\overline}
\newcommand{\wt}{\widetilde}

\newcommand{\de}{\delta}
\newcommand{\De}{\Delta}

\newcommand{\ga}{\gamma}
\newcommand{\e}{\varepsilon}

\newcommand{\om}{\omega}

\newcommand{\Si}{\Sigma}

\begin{document}

\title[The doubling map with asymmetrical holes]
{The doubling map with asymmetrical holes}
\author{Paul Glendinning and Nikita Sidorov}
\address{
School of Mathematics, The University of Manchester,
Oxford Road, Manchester M13 9PL, United Kingdom. E-mail:
p.a.glendinning@manchester.ac.uk}
\address{
School of Mathematics, The University of Manchester,
Oxford Road, Manchester M13 9PL, United Kingdom. E-mail:
sidorov@manchester.ac.uk}

\date{\today}

\begin{abstract} Let $0<a<b<1$ and let $T$ be the doubling map. Set $\mathcal J(a,b):=\{x\in[0,1] : T^nx\notin (a,b),  n\ge0\}$. In this paper we completely characterize the holes $(a,b)$ for which any of the following scenarios holds:
\begin{enumerate}
\item $\mathcal J(a,b)$ contains a point $x\in(0,1)$;
\item $\mathcal J(a,b)\cap [\de,1-\de]$ is infinite for any fixed $\de>0$;
\item $\mathcal J(a,b)$ is uncountable of zero Hausdorff dimension;
\item $\mathcal J(a,b)$ is of positive Hausdorff dimension.
\end{enumerate}
In particular, we show that (iv) is always the case if
\[
b-a<\frac14\prod_{n=1}^\infty \bigl(1-2^{-2^n}\bigr)\approx 0.175092
\]
and that this bound is sharp. As a corollary, we give a full description of first and second order critical holes introduced in \cite{SSC} for the doubling map.

Furthermore, we show that our model yields a continuum of ``routes to chaos'' via arbitrary sequences of products of natural numbers, thus generalizing the standard route to chaos via period doubling.
\end{abstract}

\maketitle

\section{Introduction}

The study of dynamical systems with holes, i.e. the characterization of points which do not fall into certain
predetermined sets under iteration by a map, poses interesting questions both about arithmetic properties
of points and their dynamical interpretation \cite{NS}. There is a growing body of work describing the effect of holes
on hyperbolic systems \cite{BKT, CM1,CM2,Troub} and expanding maps \cite{NS,SSC}. In \cite{GS} (see also \cite{ACS}) the problem
of symmetric holes about $x=\frac12$ for the doubling map $T:[0,1]\to[0,1]$ given by the formula
\[
Tx=\begin{cases} 2x, & x\in[0,1/2],\\
2x-1, & x\in(1/2, 1].
\end{cases}
\]
was studied, and in particular it was shown that if the hole is of the form $(\alpha, 1-\alpha)$ then below
a certain threshold ($\alpha =\frac{1}{3}$) the only points that do not fall into the hole are the fixed points $x=0$ and $x=1$, whilst if $\alpha$ is large enough, so the hole is small enough, then the set of points that does not fall in the hole is uncountable. This latter threshold is specified by a number related to the Thue-Morse
sequence that is familiar from the standard kneading theory of unimodal maps or Lorenz maps \cite{CE,Gbook,MT}.

The aim of this paper is to generalize these results to asymmetric holes, revealing the complex structure of the
two thresholds mentioned above, which become curves in parameter space. The methods used to describe these thresholds rely on analogous results from the theory of piecewise increasing maps of the interval, and in many ways the results presented here can be seen as 'translations' of similar results in \cite{GLT,GPTT,GH,GS,HS}.

So, let $0<a<b<1$ and put
\[
\mathcal J(a,b)=\{x\in[0,1] : T^nx\notin (a,b)\ \text{for all}\ n\ge0\}.
\]
Our goal is to study the size of the set $\mathcal J(a,b)$ for all pairs $(a,b)$. To do so, we introduce the symbolic set $\Sigma:=\{0,1\}^{\mathbb N}$ and the one-sided shift $\sigma:\Sigma\to\Sigma$ given by $\sigma(w_1,w_2,w_3,\dots)=(w_2,w_3,\dots)$. As is well known, $T\pi=\pi\sigma$, where
\[
\pi(w_1,w_2,\dots)=\sum_{n=1}^\infty w_n2^{-n},
\]
and $\pi:\Sigma\to[0,1]$ is one-to-one except for a set of sequences ending with $0^\infty$ or $1^\infty$.

We begin our investigation with a simple lemma which will help us reduce the range for $a$ and $b$.

\begin{lemma}\label{lem:simple}
\begin{enumerate}
\item If $a<1/4, b>1/2$ or $a<1/2, b>3/4$, then $\mathcal J(a,b)=\{0,1\}$.
\item If $b<1/2$ or $a>1/2$, then $\dim_H \mathcal J(a,b)>0$.
\end{enumerate}
\end{lemma}
\begin{proof}(i) Suppose $a<1/4$ and $x\in\mathcal J(a,b)$. Then $\pi^{-1}(a,b)$ contains the cylinder $[w_1=0, w_2=1]$. Hence the dyadic expansion of $x$ cannot contain $01$, which implies $x=0$ or $1$, because if the dyadic expansion of $x$ ends with $10^\infty$, it can be replaced with $01^\infty$, which lies in $(a,b)$. The case $b>3/4$ is analogous, with the cylinder $[w_1=1, w_2=0]$ instead, so we omit the proof.

\smallskip\noindent
(ii) This is essentially proved in \cite[Proposition~1.5]{SSC} but we will repeat the argument for the reader's convenience.

Assume that $[a,b]\subset (0,1/2)$ (the case $[a,b]\subset(1/2,1)$ is completely analogous). Fix $n\ge2$ and consider the following subshift of finite type:
\[
\Sigma_n=\{w\in\Sigma : w_k=0\implies w_{k+j}=1, \ j=1,\dots, n\}.
\]
When each 0 in the dyadic expansion of some $x\in(0,1)$ is succeeded by at least $n$ consecutive 1s, this means that $x\ge \frac{2^{-2}+2^{-3}+\dots+2^{-n}}{1-2^{-n-1}}$. Thus,
\[
\pi(\Sigma_n)\subset \left[\frac{\frac12-2^{-n-1}}{1-2^{-n-1}}, 1\right),
\]
which implies that $\pi(\Sigma_n)\subset \mathcal J(a,b)$ for all $n$ large enough. The topological entropy $h_{top}$ of the subshift $\sigma|_{\Si_n}$ is positive, whence $\dim_H \mathcal J(a,b)\ge \dim_H\pi(\Sigma_n)=h_{top}(\sigma|_{\Si_n})/\log2>0$.
\end{proof}

Thus, we may confine ourselves to the case $(a,b)\in (1/4,1/2)\times(1/2,3/4)$. Put
\[
D_0=\{(a,b)\in (1/4,1/2)\times(1/2,3/4) : \mathcal J(a,b)\neq\{0,1\}\}
\]
and
\[
D_1=\{(a,b)\in (1/4,1/2)\times(1/2,3/4) : \mathcal J(a,b)\ \text{is uncountable}\}.
\]

The main result of the paper is the following

\begin{thm}\label{thm:main}
We have
\[
\{(a,b)\in(0,1/2)\times (1/2,1) : \dim_H\mathcal J(a,b)>0\}= \{(a,b) : b<\chi(a)\},
\]
where $\chi$ is given by Theorem~\ref{thm:chi}.
\end{thm}

Theorem~\ref{thm:chi} requires too many preliminaries to be quoted here. Instead, we would like to present Figure~\ref{fig:b1} which may serve as a good graphic illustration of Theorem~\ref{thm:main}.

\begin{figure}[h!]\begin{center}
\includegraphics[width=0.8\textwidth]{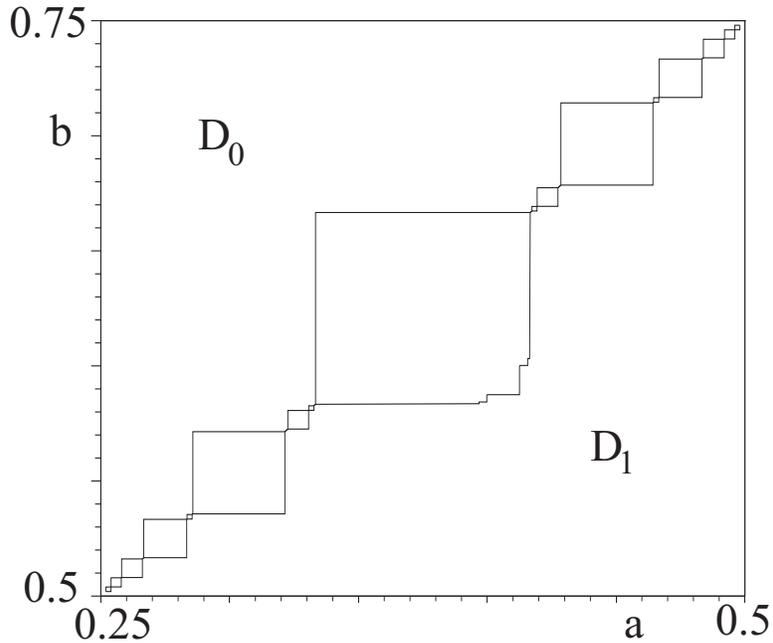}\end{center}
\vspace{-0.3cm}
\caption{\sf
 $(a,b)$-plane with $\frac{1}{4}<a<\frac{1}{2}$ and $\frac{1}{2}<b<\frac{3}{4}$ showing the
regions $D_0$ and $D_1$. The upper or left hand boundary is an approximation to $\partial D_0$ using rotation numbers $p/q$, $q=2,\dots 7$, and the lower, right hand curve is an upper bound for $\partial D_1$ using the same rotation numbers together with a second level of renormalization with rotation numbers $p/q$, $q=2,3,4$, in the $1/2$ box.}
\label{fig:b1}\end{figure}

The structure of the paper is as follows: in Section~\ref{sec-main} we give a full description of $D_0$ (Theorem~\ref{thm:D0}) and $D_1$ (Theorem~\ref{thm:main}). In Section~\ref{sec-critical} we apply our results to give a full description of first and second order critical holes for $T$ introduced in \cite{SSC}. In Section~\ref{sec-misc} we introduce an iterated function system which is naturally associated with our model and prove a claim which links our model to the well known ``transition to chaos'' phenomenon (Proposition~\ref{prop:route-chaos}). In the same section we describe the set of pairs $(a,b)$ for which $\mathcal J(a,b)$ is an uncountable set of zero Hausdorff dimension (Theorem~\ref{thm:zerodim}). Finally, in Appendix we give proofs of some technical results from Section~\ref{sec-main}.

\section{Main construction}
\label{sec-main}

\subsection{Definitions and lemmas for extremals} We need several auxiliary results first whose proofs are given in Appendix. A pair $(s,t)$ of finite words with alphabet $\{0,1\}$, where $s_1=0$ and $t_1=1$ is \emph{extremal} if and only if the inequalities
\begin{equation}\label{eq:extremal}\begin{array}{ll}
s^\infty \preceq \sigma^ks^\infty \prec t^\infty, & k=1, \dots ,|s|-1\\
s^\infty \prec \sigma^\ell t^\infty \preceq t^\infty, & \ell =1, \dots ,|t|-1
\end{array}
\end{equation}
\emph{do not} hold. In dynamical systems theory this implies that the pair $(s^\infty ,t^\infty )$ is the kneading invariant of
an expanding Lorenz map \cite{HS}.

We will use the notation $s(0,1)$ and $t(0,1)$ if the choice of alphabet is needed explicitly. Thus if $T\in\{0,1\}^n$, i.e.
$T=T(0,1)$ then $T(s,t)$ is the sequence of zeroes and ones obtained by replacing each $0$ in $T(0,1)$ by the string $s$ and
each $1$ in $T(0,1)$ by $t$. From here on any sequence denoted by $s$ or $S$ starts with zero and every sequence $t$ or $T$
starts with one. We need the following proposition:

\begin{prop}\label{prop:ext}
If $(s,t)$ and $(S,T)$ are extremal pairs, then so is $(S(s,t),T(s,t))$.
\end{prop}
\begin{proof}See Appendix.
\end{proof}

Note that we do not actually need that $(S,T)$ are well-ordered or that the length of $s$ and $t$ are equal, although this is the
case in the application we have in mind. This result is essentially part of the argument in \cite{GS}, though the proof below is considerably shorter.

Assume now that $t$ is a cyclic permutation of $s$, i.e., there exists $\ell\in\{1,\dots, N-1\}$ such that
\[
s_{\ell+1}\dots s_Ns_1\dots s_\ell=t_1\dots t_N.
\]
Put for any sequence $(n_1,n_2,\dots)\in\BN^\infty$,
\[
w(n_1,n_2,\dots)=s^{n_1}s_1\dots s_\ell s^{n_2}s_1\dots s_\ell\dots
\]
Fix $n\in\BN$ and define
\[
W_n=\overline{\{\sigma^jw(n_1,n_2,\dots) \mid n_i\in\{n,n+1\}\ \text{for all}\  i \ \text{and}\ j\ge0\}}.
\]
Clearly, $\sigma W_n=W_n$. Since we have $n_i\in\{n,n+1\}$ without restrictions, the number of 0-1 words of length~$N$ which can be extended to sequences in $W_n$ grows exponentially with $N$, whence $h_{top}(\sigma|_{W_n})>0$.

\begin{prop}\label{prop-entropy} For any 0-1 word $u\in (st^\infty,ts^\infty)$ there exists $n\in\BN$ such that
\[
W_n\cap (s^\infty, u)=\varnothing.
\]
\end{prop}
\begin{proof}See Appendix.
\end{proof}

\subsection{Introducing the functions} Following \cite{LM}, put for $a\in(1/4,1/2)$,

\begin{align*}
\phi(a)&=\sup\{b : \mathcal J(a,b)\neq\{0,1\}\},\\
\chi(a)&=\sup\{b : \mathcal J(a,b)\ \text{is uncountable}\}.
\end{align*}

We will show that
\[
D_0=\{(a,b)\in (1/4,1/2)\times (1/2,3/4) : b\le\phi(a)\}
\]
(Theorem~\ref{thm:D0})
and 
\[
D_1=\{(a,b)\in (1/4,1/2)\times (1/2,3/4) : b<\chi(a)\ \text{or}\ b\le\chi(a)\}.
\]
(Whether $\mathcal J(a,\chi(a))$ is countable or uncountable depends on $a$ in a non-trivial way -- see Proposition~\ref{prop:a-chia} below.)

The function~$\phi$ was studied in detail in \cite{AG}. For the reader's convenience we will provide an explicit formula for $\phi$ as a by-product of our investigation. Our main concern will be the function~$\chi$. Clearly, $\phi(a)\ge\chi(a)$ for all $a\in(0,1/2)$, and both functions are non-decreasing. Below we will show that the equality holds if and only if $a$ belongs to a subset of $(1/4,1/2)$ of zero Hausdorff dimension -- see Proposition~\ref{prop:equal}.

\begin{thm}\rm{(}\cite[Theorem~2]{LM}\rm{)} We have
\[
\{(a,b)\in(1/4,1/2)\times (1/2,3/4) : \dim_H \mathcal J(a,b)>0\} =\{(a,b) : b< \chi(a)\}.
\]
\end{thm}

Thus, if we give an explicit formula for $\chi$, this will lead to a full description of all pairs $(a,b)$ for which $\dim_H \mathcal J(a,b)>0$.

\subsection{Symbolic background} We need some definitions and basic results from combinatorics on words -- see \cite[Chapter~2]{Loth} for a detailed exposition. For any two finite words $u=u_1\dots u_k$ and $v=v_1\dots v_n$ we write $uv$ for their concatenation $u_1\dots u_k v_1\dots v_n$. In particular, $u^m=u\dots u$ ($m$ times) and $u^\infty=uuu\dots=\lim_{n\to\infty}u^n$, where the limit is understood in the topology of coordinate-wise convergence.

From here on by a ``word'' we will mean a word whose letters are 0s and 1s. Let $w$ be a finite or infinite word. We say that a finite  word $u$ is a {\em factor of} $w$ if there exists $k$ such that $u=w_k\dots w_{k+n}$ for some $n\ge0$. For a finite word $w$ let $|w|$ stand for its length and $|w|_1$ stand for the number of 1s in $w$. The 1-{\em ratio} of $w$ is defined as $|w|_1/|w|$. For an infinite word $w_1w_2\dots$ the 1-ratio is defined as $\lim_{n\to\infty}|w_1\dots w_n|_1/n$ (if exists).

We say that a finite or infinite word $w$ is {\em balanced} if for any $n\ge1$ and any two factors $u,v$ of $w$ of length~$n$ we have $||u|_1-|v|_1|\le1$. An infinite word is called {\em Sturmian} if it is balanced and not eventually periodic. A finite word $w$ is {\em cyclically balanced} if $w^2$ is balanced. (And therefore, $w^\infty$ is balanced.) It is well known that if $u$ and $v$ are two cyclically balanced words with $|u|=|v|=q$ and $|u|_1=|v|_1=p$ and $\gcd(p,q)=1$, then $u$ is a cyclic permutation of $v$. Thus, there are only $q$ distinct cyclically balanced words of length $q$ with $p$ 1s.

We say that a finite or infinite word $u$ is {\em lexicographically smaller than} a word $v$ (notation: $u\prec v$) if either $u_1<v_1$ or there exists $n\ge1$ such that $u_i\equiv v_i$ for $i=1,\dots, n$ and $u_{n+1}<v_{n+1}$.

For any $r=p/q\in\mathbb Q\cap(0,1)$ we define the substitution $\rho_r$ on two symbols as follows: $\rho_r(0)=\omega_r^-$, the lexicographically largest cyclically balanced word of length~$q$ with 1-ratio $r$ beginning with 0, and $\rho_r(1)=\omega_r^+$, the lexicographically smallest cyclically balanced word of length~$q$ with 1-ratio $r$ beginning with 1.

\begin{rmk}There is an explicit way to construct $\omega_r^\pm$ for any given $r$. Namely, let $r=p/q\le1/2$ have a continued fraction expansion $[d_1+1,\dots,d_n]$ with $d_n\ge2$ and $d_1\ge1$ (in view of $r\le1/2$). We define the sequence of 0-1 words given by $r$ as follows: $u_{-1}=1, u_0=0,
u_{k+1}=u_k^{d_{k+1}}u_{k-1}, \ 0\le k\le n-1$. The word $u_n$ has length~$q$ and is called the $n$th {\em standard word} given by $r$. Given an irrational $\gamma\in(0,1/2)$ with the continued fraction expansion $\ga=[d_1+1,d_2,\dots]$, the word $u_\infty$ defined as the limit of the $u_n$ is called the {\em characteristic word} given by $\ga$.

Let $w_1\dots w_q:=u_n$. Then
\begin{align*}
\omega_r^-&=01w_1\dots w_{q-2},\\
\omega_r^+&=10w_1\dots w_{q-2}.
\end{align*}
For $r\in\mathbb Q\cap(1/2,1)$ we have $\om_r^\pm=h(\om_{1-r}^\mp)$, where $h(0)=1, h(1)=0$ and $h(w_1\dots w_n)=h(w_1)\dots h(w_n)$.
\end{rmk}

\begin{example}We have $\rho_{2/5}(0)=01010,\ \rho_{2/5}(1)=10010, \ \rho_{3/5}(0)=01101,\ \rho_{3/5}(1)=10101$.
\end{example}

\subsection{The function $\phi$}
Put for any $r\in\mathbb Q\cap(0,1)$,
\[
\De(r)=[(\omega_r^-)^\infty, \omega_r^-(\omega_r^+)^\infty].
\]
(From here on we will not distinguish between $x\in[0,1]$ and its dyadic expansion, in order to simplify our notation.) It was shown by the second author in \cite{SSC} that
\begin{equation}\label{eq:S}
\mathcal S:= (1/4, 1/2) \setminus \bigcup_{r\in\mathbb Q\cap (0,1)} \De(r)
\end{equation}
has zero Hausdorff dimension. Now we give an explicit formula for $\phi$ for all $a\in(1/4,1/2)\setminus\mathcal S$.

\begin{prop}\label{prop:phi}
If $a\in\De(r)$ for some $r\in\mathbb Q\cap(0,1)$, then $\phi(a)\equiv (\omega_r^+)^\infty$. Thus, $\phi$ is piecewise constant with an infinite countable set of plateaus and the exceptional set $\mathcal S$.
\end{prop}
\begin{proof}This is a simple consequence of \cite[Theorem~7]{AG}, however, our proof is straightforward, and it should help the reader to understand better the more difficult case of the function $\chi$.

Let within this proof $s:=\omega_r^-, t:= \omega_r^+$. The pair $(s,t)$ is extremal by construction, where $\mathcal J(s^\infty, t^\infty)\supset \{T^k(s^\infty) : k\ge0\}$. In fact, we have an equality here (see \cite[Corollary~3.6]{SSC}), whence $\phi(s^\infty)\ge t^\infty$. On the other hand, by the same result, $\mathcal J(s^\infty, b)=\{0,1\}$ for any $b>t^\infty$, which implies $\phi(s^\infty)=t^\infty$.

Let now $r=p/q$ in the least terms. We have that if $a\in(s^\infty, st^\infty]$, then $T^q(a)\in[s^\infty, t^\infty]$, whence $\mathcal J(a,b)=\{0,1\}$ for any $b>t^\infty$, i.e., $\phi(a)$ cannot be larger than $\phi(s^\infty)$ for $a$ in this range.
\end{proof}

As a corollary we obtain a full description of the set $D_0$.

\begin{thm}\label{thm:D0}
We have
\[
D_0=\{(a,b)\in (1/4,1/2)\times (1/2,3/4) : b\le\phi(a)\},
\]
where $\phi$ is given by Proposition~\ref{prop:phi}.
\end{thm}
\begin{proof}It suffices to show that $\mathcal J(a,\phi(a))$ is infinite. For $a\in\mathcal S$ this follows from Propositions~\ref{prop:equal} and \ref{prop:a-chia} below. If $a\in\De(r)$ for some $r$, then the claim follows from Proposition~\ref{prop:phi} and the fact that $\mathcal J(s^\infty,t^\infty)$ is infinite (\cite[Corollary~3.6]{SSC}).
\end{proof}

\subsection{The function $\chi$} Let now $\bm r=(r_1,r_2, \dots)$ be a finite or infinite vector with each component $r_i\in\mathbb Q\cap(0,1)$. We define the sequences of 0-1 words parametrized by $\bm r$ as follows:
\begin{align*}
s_n &= \rho_{r_1}\dots \rho_{r_n}(0),\\
t_n &= \rho_{r_1}\dots \rho_{r_n}(1).
\end{align*}

\begin{example}For $r_1=1/2, r_2=1/3$ we have $s_2=\rho_{1/2}\rho_{1/3}(0)=\rho_{1/2}(010)=011001$ and $t_2=\rho_{1/2}(100)=100101$.
\end{example}

\begin{rmk}This construction appeared in \cite{BS} in connection with the study of $T$-invariant sets.
\end{rmk}

Put
\begin{align*}
\De(r_1,\dots,r_n)&=[s_n^\infty, s_nt_n^\infty],\\
\wt\De(r_1,\dots,r_n)&=[s_nt_ns_n^\infty, s_nt_n^\infty].
\end{align*}

Let $r_i=p_i/q_i$ for $1\le i\le n$ and put $Q_n=q_1\dots q_n$. Since $|\rho_r(j)|=q$ for $j=0,1$, we have $|s_n|=|t_n|= Q_n$.

\begin{lemma}Fix $(r_1,\dots,r_{n-1})$. Then for any $r_n\in\mathbb Q\cap(0,1)$ we have

\begin{equation}\label{eq:sub}
\De(r_1,\dots, r_n)\subset \wt\De(r_1,\dots, r_{n-1}).
\end{equation}
Furthermore,
\begin{equation}\label{eq:disjoint}
\De(r_1,\dots, r_n)\cap \De(r'_1,\dots, r'_n)=\varnothing\ \ \text{if}\ \ (r_1,\dots, r_n)\neq (r'_1,\dots, r'_n).
\end{equation}
Finally,
\begin{equation}\label{eq:zerohd}
\dim_H \mathcal S_n(r_1,\dots, r_{n-1})=0,
\end{equation}
where
\begin{equation}\label{eq:Sn}
\mathcal S_n(r_1,\dots, r_{n-1}):=\wt\De(r_1,\dots, r_{n-1})\setminus \bigcup_{r_n\in\mathbb Q\cap(0,1)} \De(r_1,\dots, r_n).
\end{equation}
\end{lemma}

\begin{proof} Note that since $\rho_r(0)$ is the largest among the cyclic permutations of the same word which begin with 0, $s_n$ always begins with $s_{n-1}t_{n-1}$ and, similarly, $t_n$ always begins with $t_{n-1}s_{n-1}$. Let $s_n=s_{n-1}t_{n-1}w$ and $t_n=t_{n-1}s_{n-1}w$, where $w$ is constructed from the blocks $s_{n-1}, t_{n-1}$. Thus, we have
\begin{align*}
s_{n-1}t_{n-1}w s_{n-1}t_{n-1}w\dots &\succ s_{n-1}t_{n-1}s_{n-1}^\infty,\\
s_{n-1}t_{n-1}(t_{n-1}s_{n-1}w)^\infty & \prec s_{n-1}t_{n-1}^\infty,
\end{align*}
which proves (\ref{eq:sub}).

Let us now prove (\ref{eq:disjoint}). Consider first the case $n=1$. Since $s=01w, t=10w$ for some $w$, the length of $[s^\infty, ts^\infty]$ is $1/4$, whence
\begin{align*}
st^\infty-s^\infty &= 2^{-q}(t^\infty-s^\infty),\\
t^\infty-s^\infty &=\frac14+st^\infty-s^\infty,
\end{align*}
whence the length of $[s^\infty, t^\infty]$ is $\frac{2^q}{4(2^q-1)}$. Hence the length of $\De(p/q)$ is $\frac1{4(2^q-1)}$. Therefore,

\begin{equation}\label{eq:14}
\sum_{\substack{
   q\ge 2 \\
   1\le p<q \\
   \text{g.c.d.}(p,q)=1
  }}
 |\De(p/q)|=
\sum_{q=2}^\infty \frac{\varphi(q)}{4(2^q-1)},
\end{equation}
where $\varphi$ is Euler's totient function. As is well known,
\begin{equation}\label{eq:euler}
\sum_{q=1}^\infty \frac{\varphi(q)x^q}{1-x^q}=\frac x{(1-x)^2},\quad |x|<1
\end{equation}
(see \cite[Theorem~309]{HR}). Substituting $x=1/2$ into (\ref{eq:euler}), we infer that the sum in (\ref{eq:14}) equals $1/4$, which means that the $\De(p/q)$ do not overlap for different pairs $(p,q)$.

Now suppose $n\ge2$. It suffices to consider $r_1'=r_1,\dots, r_{n-1}'=r_{n-1}$ and $r_n'>r_n$. We need to show that
\begin{equation}\label{eq:sntn}
s_n t_n^\infty \prec (s_n')^\infty.
\end{equation}
We have $s_n=S(s_{n-1},t_{n-1}), s_n'=S'(s_{n-1},t_{n-1})$ and $t_n=T(s_{n-1},t_{n-1})$. By the above, $S(0,1)T(0,1)^\infty\prec (S'(0,1))^\infty$, which means that the left-hand side has 0 and the right-hand side 1 at the first symbol where they disagree. Since $s_{n-1}$ begins with 0 and $t_{n-1}$ with 1, this implies (\ref{eq:sntn}).

To prove (\ref{eq:zerohd}), note that for $n=1$ the set $\mathcal S$ consists precisely of the points whose dyadic expansion is of the form $01w$, where $w$ is a characteristic word for some irrational $\gamma\in(0,1/2)$ -- see \cite[Section~2]{SSC}. Since for any characteristic word $w$ its prefix $w_1\dots w_N$ is balanced, the Hausdorff dimension of $\mathcal S$ is zero, in view of the fact that the number of balanced words of length~$N$ grows polynomially with $N$ -- see, e.g., \cite[Corollary~18]{Mig}.

For $n\ge2$ the set $\mathcal S_n$ is the set of points whose dyadic expansion is of the form $s_{n-1}t_{n-1}w$, where $w$ is a characteristic word with 0 replaced with $s_{n-1}$ and 1 with $t_{n-1}$. Clearly, the set which consists of such words has polynomial growth as well, whence (\ref{eq:zerohd}) follows.
\end{proof}

Put $\mathsf S_1=\mathcal S$ and
\[
\mathsf S_n=\bigcup_{\substack{
   r_1,\dots, r_{n-1} \\
   r_i\in\mathbb Q\cap (0,1),\ 1\le i\le n
  }}\mathcal S_n(r_1,\dots, r_{n-1}).
\]

\begin{prop}
The set $\mathsf S:=\bigcup_{n\ge1}\mathsf S_n$ has zero Hausdorff dimension.
\end{prop}
\begin{proof}Follows from (\ref{eq:zerohd}) and the fact that $\dim_H \bigcup_{j=1}^\infty E_j=\sup\limits_{j\in\mathbb N} \dim_H E_j$.
\end{proof}

The following key result is a generalization of \cite[Lemma~12]{GS}, where it was proved for the case $\bm r=(1/2,1/2,\dots)$. Note that our proof for the general case is completely different from that for the special case in question.

\begin{lemma}\label{lem:countable}
The set $\mathcal J(s_n^\infty, t_n^\infty)$ is infinite countable for any $r_1,\dots,r_n$.
\end{lemma}

\begin{proof}Let us first recall a well known property of cyclically balanced words. Namely, let $\{w_0,\dots,w_{q-1}\}$ be the set of cyclically balanced words of length~$q$ with $p$ 1s with $w_0\prec\dots \prec w_{q-1}$. Then there exists $p'$ such that

\begin{equation}\label{eq:pprime}
\sigma (w_j^\infty) = w_{j+p'\bmod q}^\infty,\quad 0\le j\le q-1
\end{equation}
-- see, e.g., \cite{GT}. In particular, $\sigma(w_{j-p'-1}^\infty)=w_{q-1}^\infty$ and $\sigma(w_{j-p'}^\infty)=w_0^\infty$.

We now prove the claim by induction. For $n=1$ this is \cite[Corollary~3.6]{SSC}; assume the claim to hold for all $k\le n$ and prove it for $n=k+1$. Note first that it suffices to show that for all $x\in (s_k^\infty, s_{k+1}^\infty)\cup (t_{k+1}^\infty, t_k^\infty)$, except a countable set, we have $x\notin\mathcal J(s_{k+1}^\infty, t_{k+1}^\infty)$.

Put within this proof $T_k=T^{Q_k}, \sigma_k=\sigma^{Q_k}$ and consider the set $\{\sigma_k^j(s_{k+1}^\infty) : j\ge0\}$ (whose cardinality is clearly $q_{k+1}$) and label its elements $x_0\prec\dots \prec x_{q-1}$, where $q:=q_{k+1}$. Suppose $x_{q-p'-1}=s_{k+1}^\infty$ and $x_{q-p'}=t_{k+1}^\infty$.

Put $J_i=[x_i, x_{i+1}]$ for $i=0,1,\dots, q-2$. Then (\ref{eq:pprime}) implies that for any $i\neq q-p'-1$
\[
T_k(J_i)=J_{i+p'\bmod q}.
\]
and in particular for all $i\ne q-p'-1$ there exists $j\le q-1$ such that $T_k^j(J_i)=J_{q-p'-1}$. Consequently, if $x\in [x_0,x_{q-1}]$ there exists $v\ge 0$ such that
\[
T_k^v(x)\in [x_{q-p'-1},x_{q-p'}]=[s_{k+1}^\infty, t_{k+1}^\infty].
\]
Thus, there can be only countably many $x\in(x_0, x_{q-1})$ whose trajectories do not fall into the hole $(s_{k+1}^\infty, t_{k+1}^\infty)$.

It suffices to consider $x\in (s_k^\infty, x_0)$ (the case $x\in(x_{q-1}, s_{k+1}^\infty)$ is similar). Clearly, $T_k|_{[s_k^\infty,x_0]}$ is a homeomorphism on its image with $T_k(s_k^\infty)=s_k^\infty$ and $T_k(x_0)=x_{p'}\in [x_0, x_{q-1}]$. Hence for any $x\in [s_k^\infty,x_0]$ there exists $j$ such that $T^j(x)\in [x_0, x_{q-1}]$, and we are done.
\end{proof}

Now let $\bm r=(r_1,r_2,\dots)\in(\mathbb Q\cap(0,1))^\BN$ and put
\begin{align*}
\mathfrak s(\bm r)&=\lim_{n\to\infty}\rho_{r_1}\dots\rho_{r_n}(0),\\
\mathfrak t(\bm r)&=\lim_{n\to\infty}\rho_{r_1}\dots\rho_{r_n}(1).
\end{align*}

\begin{thm}\label{thm:chi}
Any $a\in(1/4,1/2)$ falls into one of the following four categories:

\smallskip\noindent
(i) Let $s_n=\rho_{r_1}\dots\rho_{r_n}(0),\ t_n=\rho_{r_1}\dots\rho_{r_n}(1)$. We have
\begin{equation}\label{eq:thm}
\chi(a)\equiv t_ns_n^\infty\ \text{for all}\ a\in[s_n^\infty, s_nt_ns_n^\infty].
\end{equation}
Furthermore, $\chi(a)<t_ns_n^\infty$ for any $a<s_n^\infty$ and $\chi(a)>t_ns_n^\infty$ for any $a>s_nt_ns_n^\infty$, so this is an actual plateau of the function $\chi$.

\smallskip\noindent (ii) If $a\in\mathcal S$, then $\chi(a)=a+1/4$.

\smallskip\noindent (iii) If $a\in\mathcal S_n(r_1,\dots,r_{n-1})$ for $n\ge2$, then $\chi(a)=a+(1-2^{-Q_{n-1}})(t_{n-1}-s_{n-1})$.

\smallskip\noindent (iv) If there exists $(r_1,r_2,\dots)$ such that $a\in\wt\De(r_1,\dots,r_n)$ for all $n\ge1$, then $a=\mathfrak s(\bm r)$, and $\chi(a)=\mathfrak t(\bm r)$.
\end{thm}

\begin{proof}(i) Let us prove first (\ref{eq:thm}). Since $\chi$ is non-decreasing, it suffices to show that
\begin{equation}\label{eq:ineq1}
t_ns_n^\infty \preceq \chi(s_n^\infty)
\end{equation}
and
\begin{equation}\label{eq:ineq2}
t_ns_n^\infty \succeq \chi(s_nt_ns_n^\infty).
\end{equation}
Note that by Proposition~\ref{prop:ext} and induction on $n$, the pair $(s_n,t_n)$ is extremal. Hence by Proposition~\ref{prop-entropy},
\[
\dim_H\mathcal J(s_n^\infty,w)>0,\quad s_n^\infty\prec w\prec t_ns_n^\infty,
\]
which proves (\ref{eq:ineq1}). To prove (\ref{eq:ineq2}) note first that for any $w\in(s_n^\infty, s_nt_ns_n^\infty)$ there exists $j\ge0$ such that $\sigma^{jQ_n}(w)\in[s_n^\infty,t_ns_n^\infty]$ -- see Figure~\ref{fig:uninfty}.

\begin{figure}[t]
\centering \unitlength=1.3mm
\begin{picture}(40,70)(0,0)

\thinlines

\path(-10,5)(50,5)(50,65)(-10,65)(-10,5)

\dottedline(-10,5)(50,65)
\dottedline(10,5)(10,65)

\put(-13,2){$s_n^\infty$}
\put(9,2){$s_nt_n^\infty$}
\put(48,2){$t_n^\infty$}

\put(-3,2){$s_nt_ns_n^\infty$}

{\tiny \put(6,6){$\wt\De$}}

{\tiny \put(-6,6){$\De\setminus\wt\De$}}

{\Large \put(2.9,3.8){$\cdot$}}

{\Large \put(9.4,3.8){$\cdot$}}

\thicklines

\path(-10,5)(10,65)
\path(30,5)(50,65)

\thinlines

\dottedline(30,5)(30,45)
\dottedline(30,45)(3.5,45)
\dottedline(3.5,45)(3.5,5)

\put(27,2){$t_ns_n^\infty$}

\end{picture}

\caption{\sf Part of the map $T^{Q_n}$}
    \label{fig:uninfty}
  \end{figure}
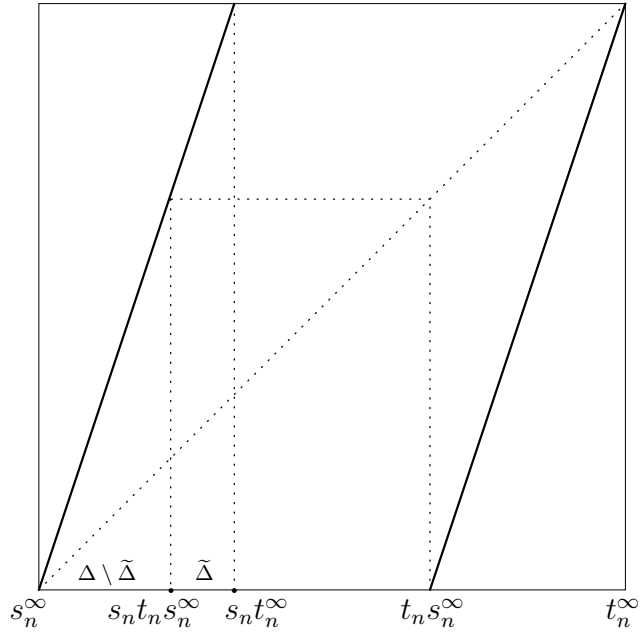

Hence $\mathcal J(s_nt_ns_n^\infty,t_ns_n^\infty)\setminus \mathcal J(s_n^\infty, t_ns_n^\infty)$ is countable. Similarly, for any $w\in(t_ns_n^\infty, t_n^\infty)$ there exists $k\ge0$ such that $\sigma^{kQ_n}(w)\in(s_n^\infty,t_ns_n^\infty)$. Therefore, $\mathcal J(s_nt_ns_n^\infty,t_ns_n^\infty)\setminus \mathcal J(s_n^\infty, t_n^\infty)$ is at most countable. Now (\ref{eq:ineq2}) follows from Lemma~\ref{lem:countable}.

To prove the second part of (i), assume first $a<s_n^\infty$. In view of (\ref{eq:disjoint}) and (\ref{eq:zerohd}), one can always find $r_n'$ such that $\De(r_1,\dots, r_{n-1}, r_n')$ lies between $a$ and $s_n^\infty$. Hence $\chi(a)\le \chi(s_n't_n'(s'_n)^\infty)=t_n'(s_n')^\infty<t_ns_n^\infty= \chi(s_n^\infty)$. Similarly, $\chi(a)>t_ns_n^\infty$ for any $a>s_nt_ns_n^\infty$.

\medskip\noindent (ii) See Proposition~\ref{prop:equal} below.

\medskip\noindent (iii) We have
\begin{align*}
\chi(a) &= a+t_ns_n^\infty-s_n^\infty \\
&= a+t_n-s_n \\
&=a+(1-2^{-Q_{n-1}})(t_{n-1}-s_{n-1}).
\end{align*}

\medskip\noindent (iv) For each $n\ge1$ we have $s_n^\infty<a<s_nt_n^\infty$, whence $\chi(s_n^\infty)\le \chi(a)\le\chi(s_nt_n^\infty)$. It suffices to recall that $\chi(s_n^\infty)=t_ns_n^\infty$ and $\chi(s_nt_n^\infty)=t_ns_nt_n^\infty$ and pass to the limit as $n\to\infty$.

\end{proof}

This proves Theorem~\ref{thm:main} stated in the Introduction. 

\begin{cor}The function $\chi$ is piecewise constant on a subset of $(1/4,1/2)$ whose complement $\mathsf S\cup\{\mathfrak s(\bm r) : \bm r\in(\mathbb Q\cap(0,1))^\BN\}$ is nowhere dense.
\end{cor}

\begin{rmk}The set $\{\mathfrak s(\bm r) : \bm r\in(\mathbb Q\cap(0,1))^\BN\}$ is studied in detail in Section~\ref{subsec:ifs} below. In particular, we show that its Hausdorff dimension is approximately $0.4732$ -- see Proposition~\ref{prop:attractor} below.
\end{rmk}

\
\begin{cor}The boundary of $D_1$ is given by the formula
\[
\partial D_1=\{(a,\chi(a)) : a\in (1/4,1/2)\}\cup \{(1-\chi(1-b),b) : b\in (1/2,3/4)\},
\]
where the values of $\chi$ can be obtained explicitly via Theorem~\ref{thm:chi}.
\end{cor}

\section{Critical holes}\label{sec-critical}

\subsection{Further properties of $\phi$ and $\chi$}

\begin{prop}\label{prop:equal}
We have $\phi(a)=\chi(a)$ for $a\in(1/4,1/2)$ if and only if $a\in\mathcal S$.
\end{prop}
\begin{proof}Let first $a\in\mathcal S$; then $a$ can be approximated by the intervals $\De(r)$ with an arbitrary precision. On each $\De(r)=[s^\infty,st^\infty]$ we have by Proposition~\ref{prop:phi} and Theorem~\ref{thm:chi}, $\phi(a)\equiv t^\infty$ and $ts^\infty\le\chi(a)\le t^\infty$. By making $|t|=q$ sufficiently large, we conclude that $\chi(a)\ge\phi(a)-\de$ for any given $\de>0$, which proves the claim.

Now suppose $a\notin\mathcal S$; then $a\in\De(r)$ for some $r$. Again, $\phi(a)=t^\infty$ for any $a\in[s^\infty, st^\infty]$, whence, in view of (\ref{eq:S}) and the monotonicity of $\chi$, it suffices to show that
\[
\chi(st^\infty)=tst^\infty\prec t^\infty.
\]
In turn, this is a consequence of $\chi(sts^\infty)=ts^\infty$ and the reciprocity of $\chi$.
\end{proof}

\begin{lemma}\label{lem:tninfty-sninfty}
Let $n\ge2, \bm r=(r_1,\dots, r_n)$ and $s_n, t_n$ as above. Then the length of the interval $[s_n^\infty, t_n^\infty]$ is equal to
\[
\frac1{4(1-2^{-Q_n})}\cdot\prod_{j=1}^{n-1}(1-2^{-Q_j}),
\]
where, as above, $Q_j=q_1\dots q_j$.
\end{lemma}
\begin{proof}Since $s_n=s_{n-1}t_{n-1}w$ and $t_n=t_{n-1}s_{n-1}w$, we have
\begin{align*}
t_n-s_n &=t_{n-1}s_{n-1}-s_{n-1}t_{n-1}\\
&=(1-2^{-Q_{n-1}})(t_{n-1}-s_{n-1}),
\end{align*}
in view of $|t_{n-1}|=|s_{n-1}|=Q_{n-1}$. Consequently,
\begin{equation}\label{eq:tn-sn}
t_n-s_n = (1-2^{-Q_{n-1}})\dots (1-2^{-Q_1})(t_1-s_1) =\frac14\cdot\prod_{j=1}^{n-1}(1-2^{-Q_j}).
\end{equation}
Now the claim follows from
\[
t_n^\infty-s_n^\infty = (1-2^{-Q_n})^{-1}(t_n-s_n).
\]
\end{proof}

\begin{prop}
We have
\begin{equation}\label{eq:phi-ineq}
a+\frac14\le \phi(a) \le a+\frac13,\quad a\in\left(\frac14,\frac12\right)
\end{equation}
and
\begin{equation}\label{eq:chi-ineq}
a+1-2a_*\le \chi(a) \le a+\frac14,\quad a\in\left(\frac14,\frac12\right),
\end{equation}
where $a_*=\lim_{n\to\infty}\rho_{1/2}^n(0)\approx0.412454$, i.e., the dyadic expansion of the Thue-Morse sequence sometimes called the {\em Thue-Morse constant} \cite{TMC}.

Furthermore, all these bounds are sharp.
\end{prop}
\begin{proof}Let us begin with (\ref{eq:phi-ineq}). We know that $\phi(a)=a+1/4$ for all $a\in\mathcal S$, so consider $a\in \Delta(r)$ for some $r$. Put, as usual, $s=\omega_r^-, t=\omega_r^+$. By Proposition~\ref{prop:phi}, $\phi(a)\equiv t^\infty$ for all $a\in\Delta(r)=[s^\infty,st^\infty]$, whence
\begin{align*}
\inf_{a\in\Delta(r)} (\phi(a)-a)&=\phi(st^\infty)-st^\infty =t^\infty-st^\infty\\
&=t-s=1/4,
\end{align*}
and
\begin{align*}
\sup_{a\in\Delta(r)}(\phi(a)-a)&\le\phi(s^\infty)-s^\infty =t^\infty-s^\infty\\
&=\frac{2^q}{4(2^q-1)}.
\end{align*}
Clearly, the right-hand side has the maximum equal to $1/3$ at $q=2$.

\medskip
\noindent Now let us prove (\ref{eq:chi-ineq}). Since the complement of $\bigcup_{\bm r} \De(\bm r)$ is nowhere dense, it suffices to study the minima and maxima of $\chi$ on the $\De(\bm r)$. In view of Theorem~\ref{thm:chi} and (\ref{eq:tn-sn}), we have for $a\in\De(\bm r)$,
\begin{align*}
\chi(a)-a&\le t_ns_n^\infty-s_n^\infty=t_n-s_n\\
&=\frac14\cdot\prod_{j=1}^{n-1}(1-2^{-Q_j})\le\frac14,
\end{align*}
and
\begin{align*}
\chi(a)-a&\ge t_ns_n^\infty-s_nt_ns_n^\infty=t_ns_n-s_nt_n\\
&=\frac14(1-2^{-Q_n})\cdot\prod_{j=1}^{n-1}(1-2^{-Q_j})\\
&=\frac14\cdot\prod_{j=1}^{n}(1-2^{-Q_j})\ge \frac14\cdot\prod_{j=1}^{n}\bigl(1-2^{-2^j}\bigr).
\end{align*}
(We use the fact that $Q_j\ge 2^j$ with the equality only if $q_j\equiv2$ for all $j$, which corresponds to $r_j\equiv1/2$.)

Therefore,
\[
\chi(a)-a\ge \frac14\cdot\prod_{j=1}^\infty\bigl(1-2^{-2^j}\bigr)=1-2a_*,
\]
with the equality at $\bm r=(1/2, 1/2, \dots)$, i.e., at $a=a_*$.
\end{proof}

\subsection{First and second order critical holes} Recall the following definitions from \cite{SSC}.

\begin{Def}
 We say that $(a_0,b_0)$ is a {\em first order critical hole} (FOCH) if the following conditions are satisfied:
\begin{enumerate}
\item for any hole $(a,b)$ such that $a<a_0, b>b_0$ we have $\mathcal J(a,b)=\{0,1\}$;
\item for any hole $(a,b)$ such that $a>a_0, b<b_0$ we have $\mathcal J(a,b)\neq\{0,1\}$.
\end{enumerate}
\end{Def}

\begin{example} An interval $(1/3,b)$ is a FOCH if and only if $b\in[7/12,2/3]$. This can be easily proved by hand but also follows from Theorem~\ref{thm:foch} below with $r=1/2$.
\end{example}

\begin{Def}
 We say that $(a_0,b_0)$ is a {\em second order critical hole} (SOCH) if the following conditions are satisfied:
\begin{enumerate}
\item for any hole $(a,b)$ such that $a<a_0, b>b_0$ the set $\mathcal J(a,b)$ is (finite or infinite) countable;
\item for any hole $(a,b)$ such that $a>a_0, b<b_0$ we have that $\dim_H\mathcal J(a,b)>0$.
\end{enumerate}
\end{Def}

As an application of our results on $\phi$ and $\chi$, we can now fully describe all first and second order critical holes for the doubling map. Note first that $(a,\phi(a))$ is a FOCH for any $a\in(0,1/2)$ by definition; however, it is possible for $(a,b)$ to be a FOCH for $b<\phi(a)$ if $a\notin\mathcal S$.

\begin{thm}\label{thm:foch}
Each FOCH for the doubling map is one of the following:
\begin{itemize}
\item $(a,1/2)$ or $(1/2,1-a)$ for any $a\in(0,1/4]$;
\item $(a,a+1/4)$, where $a\in\mathcal S$;
\item $(s^\infty,b)$, where $s=\omega_r^-$ for some $r\in\mathbb Q\cap(0,1)$ and $b\in[ts^\infty, t^\infty]$ with $t=\omega_r^+$;
\item $(a,t^\infty)$ with $a\in(s^\infty,st^\infty]$.
\end{itemize}
Consequently, the length of each FOCH can take an arbitrary value between $1/4$ and $1/2$.
\end{thm}

\begin{proof}The case $a\le1/4$ is covered by Lemma~\ref{lem:simple}~(ii), so we assume $a\in(1/4,1/2)$. Suppose first that $a\notin\mathcal S$. Then there exists $r\in\mathbb Q\cap(0,1)$ such that $a\in[s^\infty, st^\infty]$. Recall that by Proposition~\ref{prop:phi}, $\phi(a)=t^\infty$.

Let first $a\neq s^\infty$ and let $(a,b)$ be a FOCH. Then $b$ cannot be smaller than $t^\infty$, otherwise there exists $\e>0$ such that $\mathcal J(a-\e,b+\e)\neq\{0,1\}$. On the other hand, if $b$ is larger than $t^\infty$, there exists $\e>0$ such that $\mathcal J(a+\e,b-\e)=\{0,1\}$, which contradicts $(a,b)$ being a FOCH.

If $a=s^\infty$, then, similarly, $b$ cannot exceed $t^\infty$. Suppose $b$ is less than $ts^\infty$; then by Proposition~\ref{thm:chi}, there exists $\e>0$ such that $\mathcal J(a-\e,ts^\infty-\e)$ is uncountable, which is a contradiction.

Finally, suppose $a\in\mathcal S$ and $b<a+1/4$. Then for any $\e>0$ there exists $r$ such that the interval $[s^\infty, st^\infty]$ is at a distance less than $\e$ from $a$ with $st^\infty<a$. Thus, if $(a,b)$ were a FOCH, this would contradict the first part of our proof. The case $b>a+1/4$ is similar, so we omit the proof.
\end{proof}

\begin{thm}Each SOCH is one of the following:
\begin{itemize}
\item $(a,1/2)$ or $(1/2,1-a)$ for any $a\in(0,1/4]$;
\item $(a,a+1/4)$, where $a\in\mathcal S$;
\item $(a,a+(1-2^{-Q_{n-1}})(t_{n-1}-s_{n-1}))$, where $a\in\mathcal S_n(r_1,\dots, r_{n-1})$ for some $(r_1,\dots,r_{n-1})\in(\mathbb Q\cap(0,1))^{n-1}$ and $\mathcal S_n(r_1,\dots, r_{n-1})$ is given by (\ref{eq:Sn});
\item $(s_n^\infty,b)$, where $s_n=\rho_{r_1}\dots\rho_{r_n}(0)$ and $b\in[t_ns_nt_n^\infty, t_n^\infty]$ with $t_n=\rho_{r_1}\dots\rho_{r_n}(1)$;
\item $(a,t_ns_n^\infty)$ with $a\in(s_n^\infty,s_nt_ns_n^\infty]$.
\item $(\mathfrak s(\bm r), \mathfrak t(\bm r))$ for some $\bm r\in(\mathbb Q\cap(0,1))^\BN$.
\end{itemize}
 \end{thm}

\begin{proof}Again, assume first that $a\notin \mathcal S_n(r_1,\dots, r_{n-1})$, which means $a\in[s_n^\infty,s_nt_ns_n^\infty]$. Similarly to the proof of Theorem~\ref{thm:foch}, suppose first that $a>s_n^\infty$ and infer in the same manner that $b$ must be equal to $\chi(a)=t_ns_n^\infty$, since otherwise $(a,b)$ would not be a SOCH. The case $a=s_n^\infty$ is treated in the same way as above.

If $a\in\mathcal S_n(r_1,\dots,r_{n-1})$, we can approximate $a$ by some $[s_n,s_nt_ns_n^\infty]$ from below and from above and prove that $b$ cannot differ from $\chi(a)$ given by Theorem~\ref{thm:chi}. We leave the details to the reader.

Finally, if $a\in\wt\De(r_1,\dots,r_n)$ for all $n\ge1$, then $a=\mathfrak s(\bm r)$ for some $\bm r$. We have $a>s_nt_ns_n^\infty$, whence $b\ge\mathfrak t(\bm r)$. On the other hand, there exists $r_{n+1}$ such that $a<s_{n+1}^\infty$, whence $b\le t_{n+1}s_{n+1}^\infty$, and by taking the limit, $b\le \mathfrak t(\bm r)$.
\end{proof}

\begin{cor}The smallest length of a SOCH is $1-2a_*\approx 0.175092$. Consequently, $\dim_H \mathcal J(a,b)>0$ if $b-a<1-2a_*$. Assuming $a\in(1/4,1/2)$, the maximum length of a SOCH is $1/4$ and it is attained if and only if $a\in\mathcal S$.
\end{cor}

\section{Miscellaneous}\label{sec-misc}

\subsection{The underlying IFS}\label{subsec:ifs}
One can associate an iterated function system (IFS) with our model. Namely, for any $r\in(0,1)\cap\mathbb Q$, let $F_r:[0,1]\to[0,1]$ be the function which acts on the dyadic expansion of $x$ by replacing each 0 with $\omega_r^-$ and each 1 with $\omega_r^+$. The following result is straightforward.

\begin{lemma}\label{lem:ifs}
The function $F_r$ is discontinuous at each dyadic rational and continuous everywhere else. The set $F_r([0,1])$ is a Cantor set of Hausdorff dimension $1/q$.
\end{lemma}

Now for each $\bm r=(r_1,r_2,\dots)$ and each $x\in[0,1]$ we put
\[
\Phi_{\bm r}(x):=\lim_{n\to\infty} F_{r_1}\dots F_{r_n}(x).
\]
It is obvious that\footnote{In this subsection we prefer to distinguish between $w\in\Sigma$ and $\pi(w)\in[0,1]$.}
\[
\Phi_{\bm r}(x)=\begin{cases}
\pi(\mathfrak s(\bm r))\, & x<1/2,\\
\pi(\mathfrak t(\bm r)), & x>1/2,\\
\text{not defined}, & x=1/2.
\end{cases}
\]
Thus, unlike a conventional IFS (which consists of continuous functions and is usually assumed to ``contract on average''), in our model $\Phi_{\bm r}(x)$ depends on $x$, albeit in a mild way. The {\em attractor} ($=$ invariant set) of this IFS is defined in the usual way, namely,
\[
\Omega:=\left\{\pi(\mathfrak s(\bm r)) : \bm r\in(\mathbb Q\cap(0,1))^{\BN}\right\} \cup \left\{\pi(\mathfrak t(\bm r)) : \bm r\in(\mathbb Q\cap(0,1))^{\BN}\right\}.
\]

\begin{prop}\label{prop:attractor}
The set $\Omega$ is a Cantor set of Hausdorff dimension $s$ which is a unique solution of the equation
\begin{equation}\label{eq:hd-omega}
\sum_{q=2}^\infty \frac{\varphi(q)}{4^{qs}}=1,
\end{equation}
with the numerical value $s\approx 0.473223$.
\end{prop}
\begin{proof}By our construction, $F_{r_1}\dots F_{r_n}(0)\in\wt\De(r_1,\dots,r_n)$, whence
\[
\Omega\cap(0,1/2)=\bigcap_{n=1}^\infty \bigcup_{r_1,\dots,r_n}\wt\De(r_1,\dots,r_n).
\]
Let us have a look at $\wt\De(r_1,\dots, r_n)$ as a subset of $\wt\De(r_1,\dots, r_{n-1})$. We have
\begin{align}\label{eq:4.2}
\frac{|\wt\De(r_1,\dots, r_{n})|}{|\wt\De(r_1,\dots, r_{n-1})|}&=\frac{4^{-Q_{n}}(t_{n}^\infty-s_{n}^\infty)} {4^{-Q_{n-1}}(t_{n-1}^\infty-s_{n-1}^\infty)} \nonumber\\
&=4^{-q_n}\frac{(1-2^{-Q_{n-1}})^2}{1-2^{-Q_n}}\nonumber\\ &\sim 4^{-q_n},
\end{align}
Thus, although $\Omega\cap(0,1/2)$ is not literally self-similar, it is ``asymptotically self-similar'' (in view of (\ref{eq:4.2})). Hence by a version of the famous Hutchinson formula for countable IFS (see, e.g., \cite[Theorem~3.15]{MU}), $\dim_H(\Omega\cap(0,1/2))=s$, where
\begin{equation}\label{eq:4.3}
\sum_{r_n} 4^{-q_ns}=1.
\end{equation}
More precisely, $\Omega\cap(0,1/2)$ can be approximated by self-similar sets from above and below with an arbitrary precision, and then the result in question can be applied to both sequences yielding (\ref{eq:4.3}). This proves (\ref{eq:hd-omega}), since the case of $\Omega\cap(1/2,1)$ is identical.
\end{proof}

One can also endow this IFS with a natural probability measure. Namely, put for $r=p/q\in\mathbb Q\cap(0,1)$ with $1\le p<q$ and $(p,q)=1$,
\[
\mathbb P(r)=\frac1{2^q-1}.
\]
The identity $\sum_{q=2}^\infty \varphi(q)/(2^q-1)=1$ implies that $\mathbb P$ is indeed a probability measure on the rationals between 0 and 1. We denote by the same letter the product measure on $(\mathbb Q\cap(0,1))^{\BN}$ whose each multiplier is $\mathbb P$.

Now the push down measure $\nu$ supported by the attractor $\Omega$ is defined as follows:
\[
\nu(E)=\begin{cases}
\mathbb P\{\bm r : \Phi_{\bm r}(0)\in E\}, & E\subset (0,1/2),\\
\mathbb P\{\bm r : \Phi_{\bm r}(1)\in E\}, & E\subset (1/2,1)
\end{cases}
\]
and $\nu(E):=\nu(E\cap(0,1/2))+\nu(E\cap(1/2,1))$ for a general Borel set $E\subset(0,1)$. Since the Hausdorff dimension of a measure cannot exceed the Hausdorff dimension of its support, we have $\dim_H\nu<0.473223$. In particular, $\nu$ is singular.

It would be interesting to find out more about this IFS, especially about its attractor $\Omega$ and the invariant measure $\nu$. For instance, is it true that $\dim_H\nu=\dim_H\Omega$?

\subsection{An intermediate set} One can ask about the set of $(a,b)$ such that $\mathcal J(a,b)$ is infinite. However, as such this question is trivial, since if $x\in\mathcal J(a,b)\cap(0,1/2)$, then $2^{-n}x\in\mathcal J(a,b)$ for any $n\ge1$ as well, i.e., $\mathcal J(a,b)$ is automatically infinite.

To fix this, consider the map $T_{a,b}:\mathcal J(a,b)\to\mathcal J(a,b)$ defined as the restriction of $T$ to $\mathcal J(a,b)$. If $x\in[0,a]$ for $a<1/2$, then $Tx\in[0,2a]$, and $Tx\in[2b-1,1]$ for $x\in[b,1]$ with $b>1/2$. Hence the interval $[2b-1,2a]$ is the attractor for $T_{a,b}$. Now we define
\[
\wt D_0=\{(a,b) : \mathcal J(a,b)\cap[2b-1,2a]\ \text{is infinite}\}.
\]
Clearly, $D_1\subsetneq \wt D_0\subsetneq D_0$, i.e., $\wt D_0$ is an intermediate set. Similarly to the functions $\phi$ and $\chi$, we define
\begin{equation}\label{eq:psi}
\psi(a)=\sup\{b : \mathcal J(a,b)\cap[2b-1,2a]\ \text{is infinite}\},
\end{equation}
so $\wt D_0=\{(a,b) : b<\psi(a)\}$.

It is not difficult to give an explicit formula for $\psi(a)$ for all $a$, similarly to Proposition~\ref{prop:phi} and Theorem~\ref{thm:chi}. Namely, when $a\in\mathcal S$, then, in view of $\chi(a)\le\psi(a)\le\phi(a)$ and Proposition~\ref{prop:equal}, we have $\psi(a)=a+1/4$.

If $a\in[s^\infty, sts^\infty]$ for some $r$ (again, we denote $s=\omega_r^-, t=\omega_r^+$), then $\psi(a)=\chi(a)=ts^\infty$, since for $a\in[s^\infty,sts^\infty]$ and $b\in[ts^\infty,t^\infty]$, the set $\mathcal J(a,b)$ is a $q$-cycle together with its preimages, so it becomes a finite set when intersected with any interval which does not contain 0 or 1.

If $a\in[sts^\infty, st^\infty]$, we need to look at level~2. Here $\mathcal J(s_2^\infty, t_2^\infty)$ (where $s_2=s_2(r_1,r_2)$ and similarly $t_2$) contains a $q_1q_2$-cycle, and one can show that the orbit of $s_1^N s_2^\infty$ lies outside $(s_2^\infty, t_2^\infty)$ but inside the attractor for any $N\in\BN$ -- see Proposition~\ref{prop:a-chia} below. This implies $\psi(a)=t_2^\infty$ for any $a\in[s_2^\infty, s_2t_2^\infty]$.

Consequently,
\[
\frac3{16} \le \psi(a)-a\le \frac14,
\]
both bounds being sharp. (The lower one is attained at $a$ with the dyadic expansion $0110(01)^\infty$, for which $\psi(a)$ has the expansion $10(01)^\infty$.)

Thus, whilst $\phi$ is determined on level~1 and $\chi$ may require an infinite descent, $\psi$ is determined on level~1 or 2. Note also that one can replace $[2b-1,2a]$ in (\ref{eq:psi}) with $[\de,1-\de]$ for an arbitrary $\de\in(0,2b-1]$ and this will not change any value of $\psi$. We leave the details to the interested reader.

As a result, we obtain the following claim.

\begin{prop}If $b-a<\frac3{16}$, then for any fixed $\de>0$ the set $\mathcal J(a,b)\cap [\de,1-\de]$ is infinite.
\end{prop}

\subsection{Many routes to chaos} Consider first the family of symmetric holes $\{(a,1-a) : a\in(1/4,1/2)\}$. Here, as we increase $a$ from $1/4$ to $1/2$, we obtain cycles for the map $T_{a,1-a}$ in the standard Sharkovski\u{\i} order -- see \cite{ACS}. For instance, the 2-cycle appears at $a=1/3$, then we obtain a 4-cycle at $a=2/5$, etc., until we hit $a_*$ which, as we know, has the property that $\dim_H \mathcal J(a,1-a)>0$ for any $a\in(a_*,1/2)$. This is usually referred to as a {\em route to chaos via period doubling}.

Asymmetric holes provide us with a continuum of other routes to chaos. Namely, fix a parameter $\bm r=(r_1,r_2,\dots)\in(\mathbb Q\cap(0,1))^\BN$ with $r_i=p_i/q_i$ and put, as above, $Q_n=q_1\dots q_n$ and $s_n=s_n(r_1,\dots,r_n),\ t_n=t_n(r_1,\dots, r_n)$ for $n\ge1$.

Each such $\bm r$ yields its own route to chaos which formalizes the observations of \cite{GPTT} for Lorenz-like maps. They identify three different routes
to chaos. First the standard Thue-Morse type for which maps on the boundary have a countably infinite set of periodic orbits with periods related by consecutive
products (with no restriction on the possible products) realized. Second, there is the irrational rotation route to chaos (see also \cite{MT}) at which
a map on the boundary of chaos has a finite set of periodic orbits and an orbit which, for an induced map is an irrational rotation. Both these possibilities
occur at points, i.e. they are of codimension two. The third possibility is the generic case, occurring on the horizontal lines and vertical jumps of
$\chi (a)$. Here the map has a finite set of periodic orbits similar to the intermittent route of Mackay and Tresser \cite{McTr} in circle maps. From the
point of view of routes to chaos the parametrization in terms of the hole $(a,b)$ is a little unfortunate, in that the continuous curve has discontinuities
in this coordinate representation, i.e. the jumps in $\chi(a)$. It is much more natural to work in coordinates $(b-a,b+a)$ in which case
the boundary is a continuous graph as shown in Figure~\ref{fig:f3}.

\begin{figure}[h!]\begin{center}\includegraphics[width=0.8\textwidth]{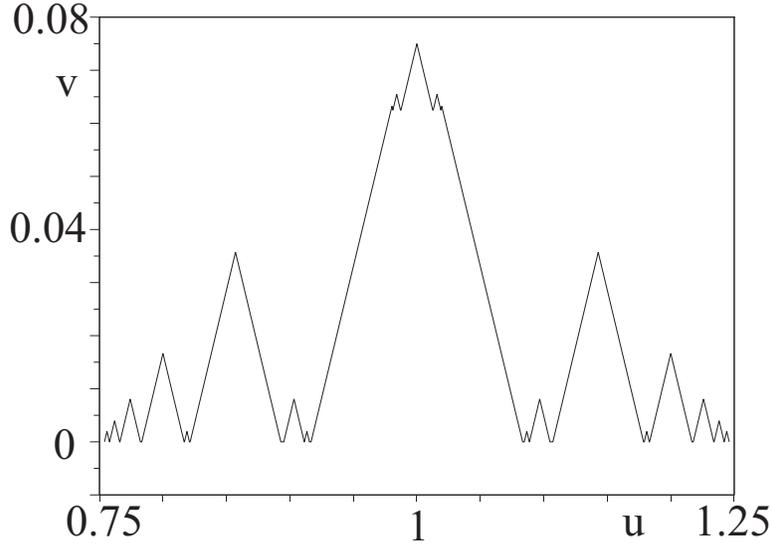}
\end{center}\vspace{-0.3cm}\caption{\sf
$(u,v)$-plane with $u=b+a$, $v=\frac{1}{4}+b-a$, showing the numerically computed approximation to  $\partial D_1$ using
$p/q$ with $q\le 7$ and the $p/q$-renormalizations with $q=2,3,4$ in the $1/2$ box.}\label{fig:f3}\end{figure}

In terms of our notation these results translate to the following statements.

\begin{prop}\label{prop:route-chaos}
If $a\in (s_n^\infty, s_{n+1}^\infty)$ and $b\in(t_{n+1}^\infty, t_n^\infty)$, then $T_{a,b}$ has a $k$-cycle if and only if $k\in\{Q_1,\dots, Q_{n+1}\}$.
\end{prop}

\begin{proof}It follows from the proof of Lemma~\ref{lem:countable} that the set $X_n:=\mathcal J(s_{n+1}^\infty,t_{n+1}^\infty)\setminus \mathcal J(s_{n}^\infty,t_{n}^\infty)$ is contained in the set of preimages of $s_{n+1}^\infty$ and $t_{n+1}^\infty$. Therefore, the only purely periodic points in $X_n$ are $s_{n+1}^\infty$ and $t_{n+1}^\infty$ themselves, both of period $Q_{n+1}$.
\end{proof}



Thus, each rectangle $(s_n^\infty,s_{n+1}^\infty)\times(t_{n+1}^\infty, t_n^\infty)$ in the $(a,b)$-plane is frequency locked and can be perceived as an `Arnold tongue' (see \cite{McTr} and \cite{GPTT}).

If we treat $a$ and $b$ as two material points moving towards each other in such a way that whenever $a=s_n^\infty$, necessarily $b=t_n^\infty$ (and arbitrary speeds between these milestones), then we obtain the cycles for $T_{a,b}$ in the following order: $q_1, q_1q_2, q_1q_2q_3$, etc. -- until $a=\mathfrak s(\bm r)$ and $b=\mathfrak t(\bm r)$. From this point on, the map $T_{a,b}$ becomes chaotic, in the spirit of \cite{GPTT}.

It would be interesting to construct meaningful analogues of the classical Sharkovski\u{\i} order for each $\bm r$. (The classical one corresponds to $\bm r=(1/2,1/2,\dots)$.) One possible way to do it could be to preserve the speed ratio for $a$ and $b$ at $\mathfrak s(\bm r)$ and $\mathfrak t(\bm r)$ respectively and continue at the same ratio inside the chaotic region.

We finish the section with a detailed study of the case $b=\chi(a)$.

\begin{prop}\label{prop:a-chia}
The set $\mathcal J(a,\chi(a))$ is uncountable of zero Hausdorff dimension if and only if $a\in\mathcal S$ or $a=\mathfrak s(\bm r)$ for some $\bm r\in(\mathbb Q\cap(0,1))^\BN$.
\end{prop}
\begin{proof}
The fact that $\dim_H\mathcal J(a,\chi(a))=0$ follows immediately from \cite[Theorem~2]{LM}.

Assume first $a\in\mathcal S$ and put $X_\ga=\ov{\{T^na : n\ge0\}}$. As was shown in \cite[Section~2]{SSC}, the (symbolic) set $X_\ga$ is the {\em Sturmian system} given by $\ga$, i.e., the set of all Sturmian sequences with the 1-ratio $\ga$. Furthermore, the dyadic expansion of $a$ is the largest element of $X_\ga$ which begins with 0, while $\chi(a)=a+1/4$ is the smallest element which begins with 1. Hence $X_\ga\cap(a,\chi(a))=\varnothing$, i.e., $\mathcal J(a,\chi(a))\supset X_\ga$. Now the claim of the theorem follows from the well known fact that $X_\ga$ has the cardinality of the continuum -- see, e.g., \cite[Chapter~2]{Loth}.

If $a\in [s_n^\infty, s_nt_ns_n^\infty]$ for some $(r_1,\dots,r_n)$, then $\mathcal J(a,\chi(a))$ is countable by Lemma~\ref{lem:countable}. By monotonicity, the same is true for any $a\in\mathcal S_n(r_1,\dots, r_{n-1})$.

Finally, assume $a=\mathfrak s(\bm r)$. Then $\chi(a)=\mathfrak t(\bm r)$, and we claim that $\mathcal J(\mathfrak s(\bm r), \mathfrak t(\bm r))$ is uncountable. More precisely,
\begin{equation}\label{eq:limitset}
s_1^{i_1}s_2^{i_2}\ldots \in \mathcal J(\mathfrak s(\bm r), \mathfrak t(\bm r)),\quad i_m\ge1,\ m\ge1.
\end{equation}
To prove this, note that by Lemma~\ref{lem:one} below,
\begin{equation}\label{eq:ineq3}
\sigma^k(s_n)s_n^\infty\prec s_n^\infty, \quad \sigma^k(t_n)s_n^\infty\prec s_n^\infty
\end{equation}
for any $n\ge1$ and any $k\in\{1,\dots, Q_n-1\}$, provided the left-hand side of (\ref{eq:ineq3}) begins with 0. We need to show that
\[
\sigma^k(s_1^{i_1}s_2^{i_2}\ldots)\prec \mathfrak s(\bm r)
\]
if the left-hand side begins with 0 and $\succ \mathfrak t(\bm r)$ otherwise for any $k\ge0$. Both cases are similar, so we assume the left-hand side to begin with 0.

If $k\le (i_1-1)q_1$, then by (\ref{eq:ineq3}), it is less than $s_1^\infty\prec \mathfrak s(\bm r)$. If $(i_1-1)q_1<k<i_1q_1$, then by the same inequality and in view of the fact that $s_2$ consists of blocks $s_1$ and $t_1$, the left-hand side is less than $s_2^\infty\prec \mathfrak s(\bm r)$. Applying the same argument for larger values of $k$ yields (\ref{eq:limitset}), whence $\mathcal J(\mathfrak s(\bm r), \mathfrak t(\bm r))$ has the cardinality of the continuum.
\end{proof}

As an immediate corollary we obtain

\begin{thm}\label{thm:zerodim}
The set $\mathcal J(a,b)$ is uncountable of zero Hausdorff dimension if and only if $a\in\mathcal S$ and $b=a+1/4$ or $a=\mathfrak s(\bm r)$ and $b=\mathfrak t(\bm r)$ for some $\bm r\in(\mathbb Q\cap(0,1))^\BN$.
\end{thm}

\section{Appendix}

\subsection{Proof of Proposition~\ref{prop:ext}}
We start with a preparatory lemma which generalizes the reflection property proved in our earlier paper \cite[Lemma~7]{GSid}:

\begin{lemma}\label{lem:one}
Suppose that $(s,t)$ is an extremal pair with $|s|=N$. If
\begin{equation}\label{eq:shiftj}
\sigma^js^\infty \prec s^\infty\quad {\rm for ~some}\quad  j\in\{1, \dots , N-1\}
\end{equation}
and
\begin{equation}\label{eq:equal}s_{j+1}\dots s_N=s_1\dots s_{N-j}\end{equation}
then $s_{N-j+1}=1$. An analogous statement holds for shifts of $t$.
\end{lemma}

\begin{proof}
Suppose the assumptions (\ref{eq:shiftj}) and (\ref{eq:equal}) of the lemma hold but that $s_{N-j+1}=0$. Then since the first $N-j$ symbols of
 $\sigma^js^\infty$ and $s^\infty$ are equal, applying $\sigma^{N-j}$ on both sides of (\ref{eq:shiftj}) does not change the inequality and so
 \begin{equation}
 \sigma^{N-j}(\sigma^js^\infty)=s^\infty\prec \sigma^{N-j}s^\infty=s_{N-j+1}\dots .
 \end{equation}
But if  $s_{N-j+1}=0$ then this implies that $s^\infty\prec \sigma^{N-j}s^\infty\prec t^\infty$ for any $t$ starting with one,
contradicting the assumption that $(s,t)$ is extremal, cf. (\ref{eq:extremal}).
\end{proof}

Let
\begin{equation}\label{eq:ILIR}
I_L=[s^\infty,st^\infty],\quad I_R=[ts^\infty,t^\infty]
\end{equation}
and note that any word constructed by concatenating $s$ and $t$ is in $I_L\cup I_R$, in particular $S(s,t)^\infty\in I_L$ and
$T(s,t)^\infty\in I_R$.

\begin{lemma}\label{lem:ILcap}
We have
\begin{equation}\label{eq:claim}
\sigma^kI_L\cap (s^\infty ,t^\infty)=\varnothing, \quad k=1, \dots |s|-1,
\end{equation}
with a similar equation holding for $I_R$.
\end{lemma}
\begin{proof}To prove (\ref{eq:claim}) note that the first symbols of both endpoints of $\sigma^kI_L$ are equal if $k\in\{1, \dots ,|s|-1|\}$.

\medskip
Suppose that the first symbol is one and $\sigma^kI_L$ does not satisfy the claim for this $k$. Then the left endpoint of $\sigma^kI_L$
is less than $t^\infty$, i.e. $\sigma^ks^\infty \prec t^\infty$. But since $\sigma^ks^\infty$ starts with a one this contradicts the assumption
that $(s,t)$ is extremal.

\medskip
Now suppose that the first symbol is zero and $\sigma^kI_L$ does not satisfy the claim for this $k$. Then by a similar argument
\begin{equation}\label{eq:in1}
s^\infty \prec \sigma^kst^\infty.
\end{equation}
Let $N=|s|$. If there is a difference in the first $N-k$ terms of these two sequences (non-empty since $k<N$), then $s^\infty \prec \sigma^ksA$
for any infinite word $A$ and in particular $s^\infty\prec \sigma^ks^\infty$, contradicting the extremality of $(s,t)$. Hence
the first $N-k$ symbols are the same, i.e. $s_1\dots s_{N-k}=s_{k+1}\dots s_N$. But this, together with the fact that $(s,t)$ is
extremal, means that Lemma~\ref{lem:one} applies and so $s_{N-k+1}=1$. Applying $\sigma^{N-k}$ to both sides of (\ref{eq:in1}),
noting that the first $N-k$ terms are equal, implies that
\[
\sigma^{N-k}s^\infty \prec t^\infty\]
and since $s_{N-k+1}=1$, $s^\infty\prec \sigma^{N-k}s^\infty$, and so we obtain a contradiction once again, thus establishing the claim.
\end{proof}

These lemmas make the proof of Proposition~\ref{prop:ext} relatively simple.

\medskip
\emph{Proof of Proposition~\ref{prop:ext}:} By Lemma~\ref{lem:ILcap}, $\sigma^kS(s,t)^\infty$ is either less than or equal to $s^\infty$ or greater than or equal to $t^\infty$ for $k=1,\dots |s|-1$,
and so strictly less than $S(s,t)^\infty$ or strictly greater than $T(s,t)^\infty$. This (together with the equivalent statement
for $I_R$) implies immediately that if an iterate of the shift of either $S(s,t)^\infty$ or $T(s,t)^\infty$ is between
$S(s,t)^\infty$ and $T(s,t)^\infty$, contradicting extremality, then an appropriate shift of $S(0,1)^\infty$ or $T(0,1)^\infty$
would lie between $S(0,1)^\infty$ and $T(0,1)^\infty$ contradicting the assumption that $(S,T)$ is extremal.

To see this in more detail, suppose that $S(s,t)=st^{p_1}s^{p_2}\dots t^{p_n}$ where the last $p_n$ may be zero (meaning no extra
symbols and the end is a power of $s$). Then $S(s,t)\in I_L$ and $\sigma^kS(s,t)\in \sigma^k(I_L)$ for $i=1,\dots |s|-1$ and so
by the lemma are outside the interval $(s^\infty,t^\infty)$ and hence also outside  $(S(s,t)^\infty,T(s,t)^\infty)$ as required for extremality.
Thus the first time we might get a contradiction of extremality is for $\sigma^{|s|}S(s,t)$, where
\[
\sigma^{|s|}S(s,t)^\infty=P_1(s,t)S(s,t)^\infty=t^{p_1}s^{p_2}\dots t^{p_n}S(s,t)^\infty,
\]
where $P_1(0,1)=\sigma S(0,1)$. On the other hand,
\[
S(s,t)^\infty\prec P_1(s,t)S(s,t)^\infty\prec T(s,t)^\infty
\]
if and only if
\[
S(0,1)^\infty\prec P_1(0,1)S(0,1)^\infty=\sigma S(0,1)^\infty\prec T(0,1)^\infty
\]
and such a relation would contradict the assumption that $(S,T)$ are extremal. Further iterates can be treated the same way.

\subsection{Proof of Proposition~\ref{prop-entropy}} Let $v\in W_n$; note first that if $j\le (n_1-1)N$, then (\ref{eq:extremal}) clearly yields the claim.

Assume that $k:=j-(n_1-1)N\in\{1,\dots, \ell-1\}$. Then by (\ref{eq:extremal}),
\[
\sigma^k v = s_{k+1}\dots s_N s_1\dots s_\ell\ldots \prec s^\infty \ \text{or}\ \succ t^\infty,
\]
since $s_{k+1}\dots s_Ns_1\dots s_k\prec s$ if $s_{k+1}=0$ and $\succ t$ otherwise.

Assume now that $k=\ell$. By the definition of $\ell$, we have
\[
u':=\sigma^\ell v = ts^{n_2}s_1\dots s_\ell s^{n_3}\dots
\]
If $n$ is chosen in such a way that $ts^n\prec u$, then $u'\succ u$.

Now assume $\ell<k<\ell+N$; in view of Lemma~\ref{lem:ILcap} and by continuity of the shift, if $z\prec ts^\infty$ is sufficiently close to $ts^\infty$, then $\sigma^j z$ is either less than $s^\infty$ or greater than $t^\infty$. (So, we increase our $n$ if necessary.)

Finally, let $k=\ell+N$. We have $\sigma^kv=s^{n_2}s_1\dots s_\ell s^{n_3}\dots$, i.e. we are back where we started.

\end{document}